\documentclass{amsart}

\usepackage{amssymb}
\usepackage{hyperref}
\usepackage{amsmath, amsthm}

\newtheorem{definition}{Definition}
\newtheorem{lemma}{Lemma}
\newtheorem{theorem}{Theorem}

\theoremstyle{definition}
\newtheorem*{forms}{Forms}
\newtheorem{example}{Example}

\theoremstyle{remark}
\newtheorem*{checking}{Checking}

\title{Infinite derivative's series expansion of Indefinite and Definite Integral}
\author{Voloshyn Victor}
\email{vitea3v@rambler.ru}
\date{27 October 2012}
\begin{document}
\maketitle
\begin{abstract}
In this article it is proven the existence of integration of indefinite integrals as infinite derivative's series expansion.
$\int f(x) d x = \sum_{ i=0 }^{ \infty } (-1)^{i} f^{(i)} (x)  \frac{ x^{i+1} }{ (i+1)! } + C$.
This also opens a new way to integrate a definite integral.
\end{abstract}
\section{Introduction}

Numerical integration is used in science nowadays, e.g. ~\cite{ni1} ~\cite{ni2}, but these methods are not as flexible as symbolic math is. \\
Symbolic computation of indefinite integrals was developed by Liouville, Risch and Bronstein (~\cite{si1},~\cite{si2}) , but Rish algorithm can not solve a lot of cases.\\
Approximation in integration with Maclaurin and Taylor series(~\cite{sit1}) are not widely used of their polynomial behavior. \\
Formulas of derivative's series expansion of Indefinite and Definite Integral, which were concluded from this work, allow to find approximate solution in symbolic math with controlled accuracy. \\

This could be an useful tool for integration multi-parametric functions and finding integrals in Differential Equations.\\
\section{Definitions}
\begin{definition}
Let $ i \in   \{ 0, \mathbb{N} \}$. We say that \textbf{$i$-order derivative} of $f(x)$ function is defined as 
\begin{align} 
f^{\left ( i \right )} \equiv \frac{d ^ i f}{d x ^ i}  \label{eq:def1}
\end{align}
or recursive
\begin{align}
\begin{cases}
f^{\left ( 0  \right )} = f \\
f^{\left ( i  \right )} = \frac{d f^{\left ( i-1  \right )}}{d x } ,  i \in  \mathbb{N}
\end{cases}
\end{align}
\end{definition}
\begin{definition}
Let $ i \in   \{ 0, \mathbb{N} \}$ We say that \textbf{$i$-order integral} of $f(x)$ function is defined as 
\begin{equation}
F^{\left ( i \right )} \equiv \underbrace { \int \cdots \int}_{i} f (d x)^i \label{eq:def2}
\end{equation}
or recursive
\begin{equation}
\begin{cases}
F^{\left ( 0 \right )} = f \\
F^{\left ( i \right )}= \int F^{\left ( i-1 \right )} d x ,  i \in  \mathbb{N}
\end{cases}
\end{equation}\\
\end{definition}
\section{Lemma of Integration of complex function}
\begin{lemma}
Let $u, v$ are functions. If $ \forall i \in  \left \{ 0,\mathbb{N} \right \} : \exists u^{(i)}, \exists V^{(i)}$ : exists all $i$-order derivatives of \textbf{u} function and exists all $i$-order integrals of \textbf{v} function then 
\begin{equation}
\int u d v =  \sum_{i=0}^{\infty } \left ( -1 \right )^ i u^{\left ( i \right )} V^{\left ( i \right )} \label{eq:uvT}
\end{equation}
\end{lemma}
\begin{proof}
Knowing (~\cite{manual}) that
\[
d \int f\left ( x \right )d x = f\left ( x \right ) d x 
\]
and
\[
\int d f = \int d f \frac{d x}{d x} = \int \frac{d f}{d x} d x
\]
and using  ~\eqref{eq:def1} and ~\eqref{eq:def2},  it's easy to show that following equality of integral $\int v d u$ is
\begin{equation}
\int v d u = \int \frac{d u}{d x}  d \left ( \int v d x \right )
\Longleftrightarrow
\int V^{\left ( 0 \right )} d u^{\left ( 0 \right )} = \int u^{\left ( 1 \right )} d V^{\left ( 1 \right )} \label{eq:vu2uv}
\end{equation}
Furthermore, integral with $i$-order derivative and $i$-order integra is
\begin{equation}
\int V^{\left ( i \right )} d u^{\left ( i \right )} = \int \frac{d u^{\left ( i \right )}}{d x}  d \left ( \int V^{\left ( i \right )} d x \right ) =\int u^{\left ( i+1 \right )} d V^{\left ( i+1 \right )} \label{eq:vu2uvC}
\end{equation}
Using (~\cite{manual})
\begin{equation}
\int u d v = u v  - \int v d u 
\end{equation}
and substituting ~\eqref{eq:vu2uv} into it, we get
\begin{equation}
\int u^{\left ( 0 \right )} d V^{\left ( 0 \right )} = u^{\left ( 0 \right )} V^{\left ( 0 \right )}  - \int u^{\left ( 1 \right )} d V^{\left ( 1 \right )} \label{eq:uv}
\end{equation}
and generalized case, using ~\eqref{eq:vu2uvC}:
\begin{equation}
\int u^{\left ( i \right )} d V^{\left ( i \right )} = u^{\left ( i \right )} V^{\left ( i \right )}  - \int u^{\left ( i+1 \right )} d V^{\left ( i+1 \right )} \label{eq:uvC}
\end{equation}
Let's evaluate the integral $\int u d v$ with ~\eqref{eq:uv} and ~\eqref{eq:uvC} using mathematical induction:
\begin{align}
\int u d v = u^{\left ( 0 \right )} V^{\left ( 0 \right )}  - \int u^{\left ( 1 \right )} d V^{\left ( 1 \right )} = u^{\left ( 0 \right )} V^{\left ( 0 \right )}  - u^{\left ( 1 \right )} V^{\left ( 1 \right )}  + \int u^{\left ( 2 \right )} d V^{\left ( 2 \right )} = \ldots \nonumber \\
= u^{(0)}V^{(0)} - u^{(1)}V^{(1)} + u^{(2)}V^{(2)} - u^{(3)}V^{(3)} + u^{(4)}V^{(4)} - u^{(5)}V^{(5)} + \ldots
\end{align}
Thus,
\begin{equation}
\int u d v =  \sum_{i=0}^{\infty } \left ( -1 \right )^ i u^{\left ( i \right )} V^{\left ( i \right )}
\end{equation}
\end{proof}
\begin{forms}
The formula ~\eqref{eq:uvT} could be written in different terms. Standard form:
\begin{equation}
\int u d v =  \sum_{i=0}^{\infty }  ( -1)^ i  \frac {d ^ i u}{d x ^ i } \underbrace { \int \cdots \int}_{i} v (d x)^i
\end{equation}
Form with partial expantion:
\begin{equation}
\int u d v =  \sum_{i=0}^{n-1 }  ( -1 )^ i u^{ ( i )} V^{ ( i )} + ( -1 )^ n \int u^{ ( n )} d V^{ ( n )} 
\end{equation}\\
\end{forms}
\begin{example}
Let's find integral of $\int \frac{ e ^ x }{x} d x$ : \\
By the Lemma 1 with ~\eqref{eq:uvT}
\begin{align}
\int \frac{ e ^ x }{x} d x =\begin{vmatrix} v = e^{x} , & V^ {\left ( i \right )} = e^{x}, C_i = 0,C_{\infty}=C \\ u = \frac{1}{x} , & u^{\left ( i \right )} = \left ( -1 \right )^{ i } i! \frac{1}{x ^ {i+1}}  \end{vmatrix} = e^{x}\sum_{i=0}^{\infty } \frac{i!}{x^{i+1}} + C
\end{align}\\
\end{example}
\begin{example}
Let's find integral of $\int \frac{ e ^ x }{x^ n} d x$ : \\
By the Lemma 1 with ~\eqref{eq:uvT}
\begin{align}
\int \frac{ e ^ x }{x^n} d x =\nonumber \\
& \begin{vmatrix} v = e^{x} , & V^ {\left ( i \right )} = e^{x}, C_i = 0,C_{\infty}=C \\ u = \frac{1}{x^n} , & u^{\left ( i \right )} = \left ( -1 \right )^{ i } \frac{(n-1+i)!}{(n-1)!} \frac{1}{x ^ {i+n}}  \end{vmatrix} =  & \nonumber \\ 
& \frac{e^{x}}{(n-1)!}\sum_{i=0}^{\infty } \frac{(n-1+i)!}{x^{i+n}} + C &
\end{align}\\
\end{example}
\section{Formulas of Integration}
\begin{theorem}[Main]
Let $f(x)$ is function. If $ \forall i \in \{ 0, \mathbb{N} \} : \exists f^{(i)} (x)$ exists all $i$-order derivatives of this function, then 
 \begin{align}
\int f(x) d x = \sum_{ i=0 }^{ \infty }(-1)^{i} f^{(i)} (x)  \frac{ x^{i+1} }{ (i+1)! } + C \label{eq:fT}
\end{align}\\
\end{theorem}
\begin{proof}
We can easily find integral by Lemma 1 and ~\eqref{eq:uvT}
\begin{align}
\int f(x) d x =  F^ { (1)} (x) = \begin{vmatrix} v= x  & V^{(i)} = \frac{ x^{i+1} }{ (i+1)! } , C_i = 0,C_{\infty}=C \\  u = f & u^{(i)} = f^{(i)} \end{vmatrix} =\nonumber \\ 
\sum_{ i=0 }^{ \infty }\left ( -1 \right )^{i} f^{(i)}(x) \frac{ x^{i+1} }{ (i+1)! } + C
\end{align}
\end{proof}
\begin{forms}
The formula ~\eqref{eq:fT} could be written in different terms. Standard form:
 \begin{align}
\int f(x) d x = f(x) + \sum_{ i=1 }^{ \infty }\left ( -1 \right )^{i}\frac{d ^ i f(x)}{d x ^ i} \frac{ x^{i+1} }{ (i+1)! } + C
\end{align}\\
\end{forms}
\begin{checking}
Let's check, if $( \int f(x) d x )' = f(x)$ : \\
From ~\eqref{eq:fT} :
\begin{eqnarray}
\lefteqn{ \left ( \sum_{ i=0 }^{ \infty } \left ( -1 \right )^{i} f^{(i)} \frac{ x^{i+1} }{ (i+1)! } + C \right )' =}  \\
& & =  \left | (fg)' = f'g + fg' \right |  = \nonumber \\ 
& & =  \sum_{i=0}^{\infty } \left ( -1 \right )^{i} f^ {(i)} \frac{(i+1) x^{i} }{ (i+1)! } + \sum_{i=0}^{\infty } \left ( -1 \right )^{i} f^ {(i+1)} \frac{ x^{i+1} }{ (i+1)! } \nonumber \\
 & & = f(x) + \sum_{i=1}^{\infty } \left ( -1 \right )^{i} f^ {(i)} \frac{ x^{i} }{ i! } - \sum_{i=1}^{\infty } \left ( -1 \right )^{i} f^ {(i)} \frac{ x^{i} }{ i! } \nonumber \\
& & =  f(x) 
\end{eqnarray}\\
\end{checking}
\begin{example}
Let's find integral of $\int x d x$ : \\
By the Theorem 1 with ~\eqref{eq:fT}
\begin{align}
\int x d x =  \begin{vmatrix}f^{(0)}=x\\f^{(1)}=1\\f^{(i)}=0, i>1, i \in \mathbb{N}\end{vmatrix} = xx-\frac{x^2}{2}+C = \frac{x^2}{2}+C
\end{align}\\
\end{example}
\begin{example}
Let's find integral of $\int e^x d x$ : \\
By the Theorem 1 with ~\eqref{eq:fT}
\begin{align}
\int e^x d x =  \begin{vmatrix}f^{(i)}=e^x, i \in \mathbb{N}\end{vmatrix} = e^x\sum_{ i=0 }^{ \infty } ( -1 )^{i}\frac{ x^{i+1} }{ (i+1)! }+C \label{eq:ex_ex}
\end{align}
Using Taylor series (~\cite{manual}) 
\begin{align}
\sum_{ i=1 }^{ \infty } ( -1 )^{i+1}\frac{ x^{i} }{ i! } = e^{-x}(e^x-1)
\end{align}
integral ~\eqref{eq:ex_ex} could be rewritten in
\begin{align}
\int e^x d x =  e^x e^{-x}(e^x-1) + C = e^x + C'
\end{align}\\
\end{example}
\begin{theorem}
Let $f(x)$ is function. If $ \forall i \in \{ 0, \mathbb{N}\}  : \exists f^ {(i)}(x) $ exists all $i$-order derivatives of this function, then 
\begin{align}
\int_{a}^{b} f (x) dx = \sum_{ i=0 }^{ \infty } ( -1 )^{i} \frac { b^{i+1} f^{(i)} \left(b\right) - a^{i+1} f^{(i)} (a)}{ (i+1)! }
\end{align}\\
\end{theorem}
\begin{proof}
By First Fundamental Theorem of Calculus~\cite{manual} and by Teorem 1 with ~\eqref{eq:fT} we get 
\begin{eqnarray}
\lefteqn{ \int_{a}^{b} f (x) dx = }  \nonumber \\
& & = \begin{vmatrix} \int_{a}^{b} f (x) dx = F^{(1)}(b) - F^{(1)}(a) \end{vmatrix} \nonumber \\
& & = \sum_{ i=0 }^{ \infty }\left ( -1 \right )^{i} f^{(i)} (b) \frac{ b^{i+1} }{ (i+1)! } + C - \left (  \sum_{ i=0 }^{ \infty }\left ( -1 \right )^{i} f^{(i)} (a) \frac{ a^{i+1} }{ (i+1)! } + C \right )   \nonumber \\
& & = \sum_{ i=0 }^{ \infty } ( -1 )^{i} \frac{ b^{i+1} f^{(i)} (b) - a^{i+1} f^{(i)} (a)}{ (i+1)! }
\end{eqnarray}
\end{proof}


\begin{thebibliography}{5}

\bibitem{manual}Milton Abramowitz, Irene A. Stegun, \textsl{Handbook of Mathematical Functions: with Formulas, Graphs, and Mathematical Tables},
9th printing, Dover Publications (1972).

\bibitem{ni1}Philip J. Davis,. Philip Rabinowitz, \textsl{Methods of Numerical Integration}:
Second Edition, Dover Publications, Inc. (2007).

\bibitem{ni2}V. I. Krylov, \textsl{Approximate Calculation of Integrals},
Dover Publications, Inc. (2006).

\bibitem{si1}Manuel Bronstein, \textsl{Symbolic Integration I: Transcendental Functions}, second edition,
Springer (2004).

\bibitem{si2}A. P. Prudnikov, Yu. A. Brychkov, and O. I. Marichev, \textsl{Integrals and Series}, vol. 2, Gordon
and Breach, Amsterdam (1998).

\bibitem{sit1}Alan Jeffrey, Daniel Zwillinger, \textsl{Table of Integrals, Series, and Products}, 
second edition, Academic Press, 2007

\end{thebibliography}
\end{document}